\documentclass[preprint,12pt]{elsarticle}

\usepackage{amssymb}
\usepackage{amsmath}
\usepackage{amsfonts}
\usepackage{graphicx}
\usepackage{graphics}
\usepackage{tikz}
\usepackage{titlesec}

\newtheorem{theorem}{Theorem}

\newtheorem{corollary}[theorem]{Corollary}

\newtheorem{definition}[theorem]{Definition}
\newtheorem{example}[theorem]{Example}

\newtheorem{lemma}[theorem]{Lemma}

\newtheorem{problem}[theorem]{Problem}
\newtheorem{proposition}[theorem]{Proposition}
\newtheorem{remark}[theorem]{Remark}

\def\qed{\vbox{\hrule
 \hbox{\vrule\hbox to 5pt{\vbox to 8pt{\vfil}\hfil}\vrule}\hrule}}

\usepackage{tikz}
\usepackage{tkz-graph}

\journal{xxxxxxx}

\begin{document}
\begin{frontmatter}

\title{Realizable Lists on a Class of Nonnegative Matrices}

\author{Cristina Manzaneda}
\address{Departamento de Matem\'{a}ticas, Facultad de Ciencias. Universidad Cat\'{o}lica del Norte. Av. Angamos 0610 Antofagasta, Chile.}
\ead{cmanzaneda@ucn.cl}

\author[]{Enide Andrade\corref{cor1}}
\address{CIDMA-Center for Research and Development in Mathematics and Applications
         Departamento de Matem\'atica, Universidade de Aveiro, 3810-193, Aveiro, Portugal.}
\ead{enide@ua.pt}

\author{Mar\'{\i}a Robbiano}
\address{Departamento de Matem\'{a}ticas, Facultad de Ciencias. Universidad Cat\'{o}lica del Norte. Av. Angamos 0610 Antofagasta, Chile.}
\ead{mrobbiano@ucn.cl}

\begin{abstract}

A square matrix of order $n$ with $n\geq 2$ is called \textit{permutative matrix} when all its rows (up to the first one) are permutations of precisely its first row. In this paper recalling spectral results for partitioned into $2$-by-$2$ symmetric blocks matrices 
sufficient conditions on a given complex list to be the list of the eigenvalues of a nonnegative permutative matrix are given. In particular, we study NIEP and PNIEP when some complex elements into the considered lists have no zero imaginary part. Realizability regions for nonnegative permutative matrices are obtained.
\end{abstract}


\end{frontmatter}

\section{Overview of Some Results}

The Nonnegative Inverse Eigenvalue Problem (called NIEP), consists on finding necessary and sufficient conditions on a list on $n$ complex numbers
\begin{equation}
\sigma =\left\{\lambda_1, \lambda_2,\ldots,\lambda_n\right\}
\label{list}
\end{equation}
to be the spectrum of an  $n$-by-$n$ entry-wise nonnegative matrix.
If there exists an $n$-by-$n$ nonnegative matrix $A$ with spectrum $\sigma$ (sometimes denoted by $\sigma(A)$), we will say that $\sigma$ is realizable and $A$ realizes $\sigma$ (or, that is a realizing matrix for the list).
An analogous problem is called the nonnegative symmetric eigenvalue problem (SNIEP) that consists on finding necessary and sufficient conditions on a list of $n$ complex numbers $\sigma$ to be the spectrum of a $n$-by-$n$ symmetric nonnegative matrix. In this case, if there exist such matrix $A$ then it is said that $\sigma$ is symmetrically realizable and $A$ is a symmetric realization of $\sigma$. The NIEP has a long history since its proposal by Kolmogorov \cite{kolmogorov} but was first formulated by Sule\u{\i}manova \cite{SLMNva} in $1949.$

\begin{definition}
{\rm The list $\sigma$ in (\ref{list}) is a Sule\u{\i}manova spectrum if the $
\lambda ^{\prime }s$ are real numbers, $\lambda _{1}>0\geq \lambda _{2}\geq
\cdots \geq \lambda _{n}$ and $\sum_{i=1}^{n}\lambda_i \geq 0$.}
\end{definition}

This problem attracted the attention of many authors over $50+$ years.  Although some partial results were obtained the NIEP is an open problem for $n \geq 5$. The NIEP for $n \leq 3$ was solved independently by Oliveira \cite[Theorem (6.2)]{ Oliveira} and Loewy and London in \cite{LwyLdn}. For matrices of order $n=4$ the problem was solved  in \cite{Meehan} and \cite{Mayo}. It has been studied in its general form in e.g. \cite{Boyle,Johnson,Laffey2,muitos,LwyLdn,SmgcH,NN3,GuoWen}. The SNIEP is also an attractive problem for many researchers, see for instance \cite{Fiedler,Laffey,LMc,Soules}.
When the lists are formed by $n$ real numbers and how they can occur as eigenvalues of a $n$-by-$n$  nonnegative matrix is also another interesting problem and it is called real nonnegative inverse eigenvalue problem (RNIEP). Some results concerning RNIEP can be seen in e.g. \cite{Boro,Friedland,HzlP, RS, SRM}. A recent survey on NIEP can be seen in \cite{JhonsonMarijuan}.

Concerning the NIEP there are some immediate necessary conditions on a list of complex numbers $\sigma =\left\{ \lambda _{1},\lambda _{2},\ldots ,\lambda _{n}\right\}$ to be the spectrum of a nonnegative matrix.

\begin{enumerate}

\item The list $\sigma $ is closed under complex conjugation.

\item The Perron eigenvalue $\rho=\max \left\{ \left\vert \lambda \right\vert
:\lambda \in \sigma \left( A\right) \right\} $ lies in $\sigma.$

\item $s_{k}\left( \sigma \right) =\sum\limits_{i=1}^{n}\lambda
_{i}^{k}\geq 0.$

\item $s_{k}^{m}\left( \sigma \right) \leq n^{m-1}s_{km}\left( \sigma
\right)$ for $ k,m=1,2,\ldots $
\end{enumerate}

Theory and applications of nonnegative matrices are blended in the book \cite{BergmanPlemmons}. Practical problems with applications to Markov chains, queuing networks, economic analysis, or mathematical programming and also inverse eigenvalue problems are presented there. Extensive references are included in each area. See also the book \cite{Mink} for a general theory on nonnegative matrices.

One of the most promising attempts to solve the NIEP is using constructive methods. For instance, given a class of nonnegative matrices find a large class of spectra realized by these matrices. Companion matrices of polynomials played an important role.
Let
\begin{eqnarray*}
f(x) &=& (x-\lambda_1)(x- \lambda_2) \cdots (x- \lambda_{n})\\
     &=& x^n + a_{1} x^{n-1}+ a_{2} x^{n-2}+ \cdots + a_{n}.
\end{eqnarray*}
The companion matrix of $f(x)$ is:
\begin{equation*}
C(f)= \left[
\begin{array}{ccccc}
0        &      1   & 0      & \cdots & 0 \\
0        &      0   & 1      & \ddots & \vdots \\
\vdots   &   \vdots & \ddots & \ddots & 0 \\
0        &     0    & \cdots &   0    & 1 \\
-a_{n}   & -a_{n-1} & \cdots & -a_2   & -a_1
\end{array}
\right].
\end{equation*}
Note that $A$ is nonnegative if and only if $a_i \leq 0,$ for $ i= 1, \ldots,n.$

Friedland \cite{Friedland} stated that Sule\u{\i}manova-type spectra are realizable by a companion matrix. Lowevy and London, in 1978 \cite{LwyLdn} showed that the list $(r, a+bi, a-bi),$ with $r \geq 1$, is realizable if and only if it is realizable by the matrix $\alpha I+C$, with $\alpha \geq 0$ and $C$ a nonnegative companion matrix. In \cite{LaffSmgc2}, T. Laffey and H. \v{S}migoc stated that a list $\sigma =\{\lambda_1, \ldots, \lambda_n\},$ with $\lambda_1 >0$
and $Re(\lambda_i) \leq 0, \ i =2, \ldots,n,$ is realizable if and only if it is realizable by a matrix of the form $\alpha I+C,$ where $\alpha \geq 0$ and $C$ is a companion matrix with trace zero. Rojo and Soto in \cite{RojoSoto} found sufficient conditions for the realizability of spectra by nonnegative circulant matrices. In \cite{KiHang} the authors studied realizability by nonnegative integral matrices.
Leal-Duarte and Johnson in \cite{LealJonhson} solved the NIEP for the case the realizing matrix is an arbitrary nonnegative diagonal matrix added to a nonnegative matrix whose graph is a tree. G. Soules in \cite{Soules}, construct a symmetric matrix $N$ having specified list of eigenvalues and give conditions on the maximal eigenvalue such that $N$ is nonnegative.
A problem proposed by T. Laffey and H. \v{S}migoc is to find good classes of matrices to study realizability and  many authors tried to find good classes of matrices to study this type of problems. Here we will focus on permutative matrices.
The summary of the paper is the following: At Section 1 we review some state of the art related with NIEP. At Section 2 we recall some recent results related with permutative matrices and at Section 3 we present some known spectral results for some structured matrices. The new results appear at Section 4 where are studied  NIEP and PNIEP (we call the problem as PNIEP when the NIEP involves permutative matrices),  when some complex elements into the considered lists have no zero imaginary part.  It is presented sufficient conditions on a complex list of four elements to be realized by a permutative matrix. The class of circulant  and and skew circulant matrices and their properties play an important role at Section 5 and some results are derived.

\section{A brief History on Permutative Matrices and Recent Results}
In this section we recall some useful tools from recent published results that will be used throughout the text. Some auxiliary results from \cite{PP}, and some recent definitions from the literature are recalled here. The next definition was firstly presented in \cite{PP} and it is the definition of permutative matrix. In fact, this term was given by C. R. Johnson (see footnote at \cite{PP}).

\begin{definition}
\label{ept} \cite{MAR}
Let $\mathbf{\tau }=\left(\tau _{0}, \tau _{1},\ldots ,\tau _{n-1}\right) $ be an $n$-tuple whose components are permutations in the symmetric group\ $\mathcal{S}_{n}$, with $\tau _{0}=id$.\ Let $\mathbf{a=}\left( a_{1},\ldots ,a_{n}\right) \in
\mathbb{C}^{n}$. Define the row-vector,
\begin{equation*}
\tau _{j}\left( \mathbf{a}\right) =\left( a_{\tau _{j}\left( 1\right)
},\ldots ,a_{\tau _{j}\left( n\right) }\right)
\end{equation*}%
and consider the matrix
\begin{equation}
\tau \left( \mathbf{a}\right) =%
\begin{pmatrix}
\tau _{1}\left( \mathbf{a}\right)   \\
\tau _{2}\left( \mathbf{a}\right)  \\
\vdots \\
\tau _{n-1}\left( \mathbf{a}\right)  \\
\tau _{n}\left( \mathbf{a}\right)
\end{pmatrix}
.  \label{permut}
\end{equation}
An $n$-by-$n$ matrix $A$ is called \textit{permutative} if
$A=\tau \left( \mathbf{a}\right) $ for some $n$-tuple $\mathbf{a}$.
\end{definition}




Ranks of permutative matrices were studied by Hu et al \cite{Hu}. The authors focus on identifying circumstances under which square permutative matrices are rank deficient.

In \cite{PP}, dealing with RNIEP, P. Paparella cosidered cases of realizable spectra when a realizing matrix can be taken to have a specific form, that is, to be a permutative matrix. Paparella raised the question when any realizable (real)  list can be realized by such a matrix or a direct sum of permutative matrices. The author showed that for $n \leq 4$ this is always possible. Moreover, it is shown in \cite{PP} that if the list $\sigma$ contains one positive number and is realizable then it can be realized by a permutative matrix, and thus explicit permutative matrices which realize Sule\u {\i}manova spectra were found. The author used a construtive proof. Loewy  in \cite{Lwy} showed than in general the answer to the question posed by P. Paparella is no.
Loewy \cite{Lwy} resolved this problem in the negative by showing that the list $ \sigma= \left( 1,\frac{8}{25}+\frac{\sqrt{51}}{50},\frac{8}{25}+\frac{\sqrt{51}}{50},-\frac{4}{5}, -\frac{21}{25}\right) $ is realizable but cannot be realized by a permutative matrix or by a direct sum of permutative matrices.

Recently, in \cite{MAR} the spectra of a class of permutative matrices were studied. In particular, spectral results for matrices partitioned into $2$-by-$2$ symmetric blocks were presented and, using these results sufficient conditions on a given list to be the list of eigenvalues of a nonnegative permutative matrix were obtained and the corresponding permutative matrices were constructed. We recall here some useful definitions from \cite{MAR} that will be used in the sequel.

\begin{definition}\cite{MAR}
If $A$ and $B$ are permutative by a common vector
$\mathbf{\tau }=\left( \tau _{0},\ldots ,\tau _{n-1}\right)$
then they are called \textit{permutatively equivalent}.
\end{definition}

\begin{definition}\cite{MAR}
Let $\varphi \in \mathcal{S}_{n}$ and the $n$-tuple $\mathbf{\tau }=\left(
id,\varphi ,\varphi ^{1},\ldots ,\varphi ^{n-1}\right) \in \left(
\mathcal{S}_{n}\right) ^{n}.$ Then a $\mathbf{\tau }$-permutative matrix is called $
\varphi $\textit{-permutative}.
\end{definition}

It is clear from the definitions that two
$\varphi $-permutative matrices are permutatively equivalent matrices.

\begin{remark}
If permutations are regarded as bijective maps from the set
\begin{equation*}\left\{ 0,1,\cdots, n-1\right\}\end{equation*} on to itself, then a circulant (respectively, left circulant) matrix
is a $\varphi $-permutative matrix
where
$\varphi \left( i \right)\equiv i-1(\mbox{mod}\ {n})$ (resp. $\varphi \left( i \right)\equiv i+1(\mbox{mod}\ {n})$).
\end{remark}


Throughout the paper and if no misunderstanding arise, the symbol $i$ stands for the complex square root of $-1$ and the row sub-index of the $(i,j)$-entry of some considered matrix. Note that the lists considered along the paper are equivalent (up to a permutation of its elements).

\section{Known Spectral Results on some Structured Matrices}

In this section we exhibit spectral results for matrices that are partitioned into $2$-by-$2$
symmetric blocks. The following results were proven in \cite{MAR}. The next theorem is valid in an algebraic closed field $K$ of characteristic $0$. For instance, $K=\mathbb{C}$.

\begin{theorem}
\label{main}\cite{MAR}
Let $K$ be an algebraically closed field of characteristic $0$ and suppose
that $A=\left( A_{ij}\right) $ is a block matrix of order $2n$, where
\begin{equation}
A_{ij}=%
\begin{pmatrix}
a_{ij} & b_{ij} \\
b_{ij} & a_{ij}
\end{pmatrix}
\text{,\ }a_{ij}\text{,\ }b_{ij}\in K.
\end{equation}
If
\[
s_{ij}=a_{ij}+b_{ij},\ 1\leq i,j\leq n
\]%
and%
\[
c_{ij}=a_{ij}-b_{ij},\ 1\leq i,j\leq n
\]%
then
\[
\sigma \left( A\right) =\sigma \left( S\right) \cup \sigma \left( C\right)
\]%
where
\[
S=\left( s_{ij}\right) \text{ and }C=\left( c_{ij}\right) .
\]
\end{theorem}

\begin{theorem}
\label{main2} \cite{MAR} Let $S=\left(  s_{ij}\right)  $ and $C=\left( c_{ij}\right)$
be matrices of order $n$ whose spectra (counted with their
multiplicities) are $\sigma (S)=\{\lambda_{1},\lambda_{2},\ldots
,\lambda_{n}\}  \ $and$\ \sigma (C)=\{\mu_{1},\mu_{2},\ldots
,\mu_{n}\}  $, respectively. Let $0\leq \gamma \leq 1$. If
\begin{equation}
\left\vert c_{ij}\right\vert \leq s_{ij}, 1\leq i,j\leq
n,\label{mayorize}
\end{equation}
(or equivalently if $S$, $S+C$ and $S-C$ are nonnegative matrices), then the matrices $\frac{1}{2}\left(  S+\gamma C\right)$ and $\frac{1}{2}\left(
S-\gamma C\right)$ are nonnegative and the nonnegative matrices
\begin{equation}
M_{\pm\gamma}=\left( M_{{ij}_{\pm\gamma}}\right)  ,\text{with \,} M_{{ij}_{\pm\gamma}}=
\begin{pmatrix}
\frac{s_{ij} \pm\gamma c_{ij}}{2} & \frac{s_{ij}\mp\gamma c_{ij}}{2}\\
\frac{s_{ij}\mp\gamma c_{ij}}{2} & \frac{s_{ij}\pm\gamma c_{ij}}{2}%
\end{pmatrix}
,\ \text{for}\ 1\leq i,j\leq n\label{byblock}%
\end{equation}
realize, respectively, the following lists
\[
\sigma(S)\cup\gamma \sigma(C):=\{ \lambda_{1},\lambda_{2},\ldots,\lambda_{n},\gamma\mu
_{1},\gamma \mu_{2},\ldots, \gamma \mu_{n}\}
\] and
\[
\sigma(S)\cup\left(-\gamma\sigma(C)\right):=\{ \lambda_{1},\lambda_{2},\ldots,\lambda_{n},-\gamma\mu
_{1},-\gamma\mu_{2},\ldots,-\gamma\mu_{n}\}.
\]

\end{theorem}

\begin{remark}
\label{important} \cite{MAR} Note that in the previous result if $S=\left( s_{ij}\right) $
and $C=\left( c_{ij}\right) $, then
\begin{equation}
M_{\pm\gamma}=
\begin{pmatrix}
\frac{s_{11} \pm \gamma c_{11}}{2} & \frac{s_{11} \mp\gamma c_{11}}{2} & \ldots & \ldots &
\frac{s_{1n} \pm \gamma c_{1n}}{2} & \frac{s_{1n} \mp\gamma c_{1n}}{2} \\
\frac{s_{11} \mp\gamma c_{11}}{2} & \frac{s_{11} \pm \gamma c_{11}}{2} & \ldots & \ldots &
\frac{s_{1n}\mp\gamma c_{1n}}{2} & \frac{s_{1n}\pm \gamma c_{1n}}{2} \\
\vdots & \vdots & \ddots & \ddots & \vdots & \vdots \\
\vdots & \vdots & \ddots & \ddots & \vdots & \vdots \\
\frac{s_{n1} \pm \gamma c_{n1}}{2} & \frac{s_{n1}\mp \gamma c_{n1}}{2} & \ldots & \ldots &
\frac{s_{nn} \pm \gamma c_{nn}}{2} & \frac{s_{nn} \mp\gamma c_{nn}}{2} \\
\frac{s_{n1} \mp\gamma c_{n1}}{2} & \frac{s_{n1} \pm \gamma c_{n1}}{2} & \ldots & \ldots &
\frac{s_{nn}\mp\gamma c_{nn}}{2} & \frac{s_{nn}\pm \gamma c_{nn}}{2}%
\end{pmatrix}%
.  \label{aspectM}
\end{equation}
\end{remark}

\begin{remark}
\label{cond}\cite{MAR} For two permutatively equivalent $n$-by-$n$ matrices $S=\left(
s_{ij}\right)  $ and $C=\left(  c_{ij}\right)  $ whose first row
are the $n$-tuples $\left(  s_{1},\ldots,s_{n}\right),  $ and
$\left(  c_{1},\ldots,c_{n}\right)  ,$ respectively, the inequalities $\left\vert
c_{ij}\right\vert \leq s_{ij}$ hold if and only if $\left\vert
c_{i}\right\vert \leq s_{i}, 1\leq i\leq n.$
\end{remark}
\begin{theorem}\label{main copy(1)}\cite{MAR}
Let $K$ be an algebraically closed field of characteristic $0$ and suppose
that $A=\left( A_{ij}\right) $ is an into block square matrix of order $2n+1$, where
\[
A_{ij}=\left\{
\begin{tabular}{cc}
$%
\begin{pmatrix}
a_{ij} & b_{ij} \\
b_{ij} & a_{ij}%
\end{pmatrix}%
$ & $1\leq i,j\leq n$ \\
$%
\begin{pmatrix}
a_{ij} \\
a_{ij}%
\end{pmatrix}%
$ & $1\leq i\leq n,\ j=n+1$ \\
$%
\begin{pmatrix}
a_{ij} & b_{ij}%
\end{pmatrix}%
$ & $i=n+1,$\ $1\leq j\leq n$ \\
$a_{ij}$ & $i=n+1,$\ $\ j=n+1.$%
\end{tabular}%
\right.
\]%
If
\[
s_{ij}=\left\{
\begin{tabular}{cc}
$a_{ij}+b_{ij}$ & $1\leq i,j\leq n$ \\
$a_{ij}$ & $1\leq i\leq n,\ j=n+1$ \\
$a_{ij}+b_{ij}$ & $i=n+1,$\ $1\leq j\leq n$ \\
$a_{ij}$ & $i=n+1,$\ $\ j=n+1$%
\end{tabular}%
\right.
\]%
and%
\[
c_{ij}=a_{ij}-b_{ij},\ 1\leq i,j\leq n.
\]%
Then
\[
\sigma \left( A\right) =\sigma \left( S\right) \cup \sigma \left( C\right),
\]%
where
\[
S=\left( s_{ij}\right) \text{ and }C=\left( c_{ij}\right) .
\]
\end{theorem}

\section{Complex lists}
In this section we study NIEP and PNIEP when some complex elements into the considered lists have no zero imaginary part.
The following result gives sufficient conditions on a complex list of four elements to be realized by a permutative matrix.

\begin{theorem} \label{first}
Let $\sigma =\left\{ \lambda _{1},\lambda _{2},\lambda _{3},\lambda
_{4}\right\}$ such that $\ \sum\limits_{i=1}^{4}\lambda_{i}\geq 0$, $\overline{\lambda }_{3}=\lambda _{4}$,  $\lambda _{1}+\lambda _{2}\geq 2 \rm{Re}\lambda _{3}$,
and
$\lambda _{1}-\lambda _{2}\geq 2\left\vert \rm{Im}\lambda _{3}\right\vert $,
then the nonegative permutative
matrix
\begin{equation}
M=
\begin{pmatrix}
\frac{\lambda _{1}+\lambda _{2}+\lambda _{3}+\lambda _{4}}{4} & \frac{%
\lambda _{1}+\lambda _{2}-\lambda _{3}-\lambda _{4}}{4} & \frac{\lambda
_{1}-\lambda _{2}-\lambda _{3}i+\lambda _{4}i}{4} & \frac{\lambda
_{1}-\lambda _{2}+\lambda _{3}i-\lambda _{4}i}{4} \\
\frac{\lambda _{1}+\lambda _{2}-\lambda _{3}-\lambda _{4}}{4} & \frac{%
\lambda _{1}+\lambda _{2}+\lambda _{3}+\lambda _{4}}{4} & \frac{\lambda
_{1}-\lambda _{2}+\lambda _{3}i-\lambda _{4}i}{4} & \frac{\lambda
_{1}-\lambda _{2}-\lambda _{3}i+\lambda _{4}i}{4} \\
\frac{\lambda _{1}-\lambda _{2}+\lambda _{3}i-\lambda _{4}i}{4} & \frac{%
\lambda _{1}-\lambda _{2}-\lambda _{3}i+\lambda _{4}i}{4} & \frac{\lambda
_{1}+\lambda _{2}+\lambda _{3}+\lambda _{4}}{4} & \frac{\lambda _{1}+\lambda
_{2}-\lambda _{3}-\lambda _{4}}{4} \\
\frac{\lambda _{1}-\lambda _{2}-\lambda _{3}i+\lambda _{4}i}{4} & \frac{%
\lambda _{1}-\lambda _{2}+\lambda _{3}i-\lambda _{4}i}{4} & \frac{\lambda
_{1}+\lambda _{2}-\lambda _{3}-\lambda _{4}}{4} & \frac{\lambda _{1}+\lambda
_{2}+\lambda _{3}+\lambda _{4}}{4}%
\end{pmatrix}
\label{reali3}
\end{equation}%
realizes $\sigma .$
\end{theorem}

\begin{proof}
Suppose that
\begin{equation*}
M=%
\begin{pmatrix}
a & b & c & d \\
b & a & d & c \\
d & c & a & b \\
c & d & b & a%
\end{pmatrix}.%
\end{equation*}%
By Theorem \ref{main}, the spectrum of $M$ is the union of the spectra of $S$
and $C$ respectively,  with
\begin{equation*}
S=%
\begin{pmatrix}
a+b & c+d \\
d+c & a+b%
\end{pmatrix}%
\text{ and }C=%
\begin{pmatrix}
a-b & c-d \\
d-c & a-b%
\end{pmatrix}%
.
\end{equation*}%
Thus,
\begin{equation*}
\sigma \left(M\right) =\left\{ a+b+c+d,a+b-c-d\right\} \cup \left\{
a-b+ic-id,a-b-ic+id\right\}.
\end{equation*}
Considering
\begin{eqnarray*}
\lambda _{1} &=&a+b+c+d \\
\lambda _{2} &=&a+b-c-d \\
\lambda _{3} &=&a-b+ic-id \\
\lambda _{4} &=&a-b-ic+id
\end{eqnarray*}
and solving we obtain,
\begin{eqnarray*}
a &=&\frac{\lambda _{1}+\lambda _{2}+\lambda _{3}+\lambda _{4}}{4} \\
b &=&\frac{\lambda _{1}+\lambda _{2}-\lambda _{3}-\lambda _{4}}{4} \\
c &=&\frac{\lambda _{1}-\lambda _{2}-\lambda _{3}i+\lambda _{4}i}{4} \\
d &=&\frac{\lambda _{1}-\lambda _{2}+\lambda _{3}i-\lambda _{4}i}{4}.
\end{eqnarray*}
The conditions in the statement imply that the entries $a,b,c$ and $d$ are nonnegative.
\end{proof}

\begin{example}
Since \ $\sigma =\left\{ 8,-6,-1+5i,-1-5i\right\} $ satisfies the conditions of Theorem \ref{first} we can obtain the permutative matrix in (\ref%
{reali3}),
\begin{equation*}
M=%
\begin{pmatrix}
0 & 1 & 6 & 1 \\
1 & 0 & 1 & 6 \\
1 & 6 & 0 & 1 \\
6 & 1 & 1 & 0%
\end{pmatrix}.
\end{equation*}
\end{example}

\begin{example}
Since $\sigma=\left\{ 8,2,3+2i,3-2i\right\} $ satisfies the conditions of Theorem \ref{first} we can obtain the permutative matrix in (\ref%
{reali3}),
\begin{equation}
M=
\begin{pmatrix}
4 & 1 & \frac{5}{2} & \frac{1}{2}\\
1 & 4 & \frac{1}{2} & \frac{5}{2}\\
\frac{1}{2} & \frac{5}{2} & 4 & 1\\
\frac{5}{2} & \frac{1}{2} & 1 & 4
\end{pmatrix}.
\end{equation}
\end{example}

\begin{remark}
The two above examples have its spectra out of the complex region
\[
\Gamma=\left\{  z\in\mathbb{C}:\operatorname{Re}z\leq0\text{ and }\left\vert
\operatorname{Im}z\right\vert \leq\left\vert \operatorname{Re}z\right\vert
\right\} ,
\]
(region presented in \cite{BoSo} for the realizability of some complex spectrum).
\end{remark}

\begin{remark}
In \cite{EG} after describing some previous results the following problem was considered:
Find a geometric representation for $\left(r,a,b\right) \in \mathbb{R}^3$ such that all lists $\{1,r,a+ib,a-ib\}$ having Perron root $1$ in the list, to be a NIEP's solution.
\end{remark}
\vspace{0.3cm}
Towards a response to this problem we propose the following result.

\begin{theorem}
Let consider the complex list $\sigma =\left\{1,r,a+ib,a-ib\right\}$ with $0\leq r\leq1$. If $|a|\leq\frac{1+r}{2}$ and $|b|\leq \frac{1-r}{2}$ then the list $\sigma$ is a realizable list whose realizing matrix is

\begin{equation}
M=
\begin{pmatrix}
\frac{ 1+r+2a}{4} & \frac{1+r-2a}{4} & \frac{1-r+2b}{4} & \frac{1-r-2b}{4} \\
\frac{1+r-2a}{4} & \frac{1+r+2a}{4} & \frac{1-r-2b}{4} & \frac{1-r+2b}{4} \\
\frac{1-r-2b}{4} & \frac{1-r+2b}{4} & \frac{1+r+2a}{4} & \frac{1+r-2a}{4} \\
\frac{1-r+2b}{4} & \frac{1-r-2b}{4} & \frac{1+r-2a}{4} & \frac{1+r+2a}{4}
\end{pmatrix}.
\end{equation}
\end{theorem}

\begin{proof}
Suppose that $\left(1,r,a+ib,a-ib\right)=\left( \lambda _{1},\lambda _{2},\lambda _{3},\lambda
_{4}\right) \in \mathbb{C}^{4}$. Then $\overline{\lambda }_{3}=\lambda _{4}$. From Theorem \ref{first} the conditions:
\[
1+r \geq 2a \qquad and \qquad 1-r\geq 2|b|
\] must be fulfilled. Considering the property of the trace we obtain: $$1+r+2a\geq 0.$$ Therefore,
$$|a|\leq \frac{1+r}{2}\qquad and \qquad |b|\leq \frac{1-r}{2}$$
which is a spectral realization region obtained in terms of $r.$

\end{proof}

\section{Circulant and skew circulant matrices}

The class of circulant matrices and their properties are introduced in \cite{Davis} and plays an important role here.
In \cite{Karner} it was presented spectral decomposition of four types of real
circulant matrices. Among others, right circulants (whose elements topple
from right to left) as well as skew right circulants (whose elements change
their sign when toppling) are analyzed. The inherent periodicity of
circulant matrices means that they are closely related to Fourier analysis
and group theory.

Let $s=\left( s_{0},s_{1},\ldots ,s_{n-1}\right) ^{T},\ c=\left(
c_{0},c_{1},\ldots ,c_{n-1}\right) ^{T}\in \mathbb{R}^{n}$ be given.

\begin{definition} \cite{Davis, Karner}
A \emph{\ real right circulant matrix} (or simply, \emph{circulant matrix}), is a matrix of the form
\begin{equation*}
S\left( s\right) =
\begin{pmatrix}
s_{0} & s_{1} & \ldots  &  & s_{n-1} \\
s_{n-1} & s_{0} & s_{1} &  & s_{n-2} \\
s_{n-2} & \ddots  & \ddots  & \ddots  & \vdots  \\
\vdots  & \ddots  & \ddots  & s_{0} & s_{1} \\
s_{1} & \ldots  & s_{n-2} & s_{n-1} & s_{0}
\end{pmatrix}
\end{equation*}
where each row is a cyclic shift of the row above to the right.
\end{definition}

The matrix $S\left( s\right)$ is a special case of a Toeplitz matrix and it is clearly determined by its first row. Therefore, if no confusion arise, the above circulant matrix is also sometimes denoted by $circ(s_0,s_1,\ldots,s_{n-1}).$
\begin{definition} \cite{Karner}
A \emph{real} \emph{skew right circulant matrix} or simply a \emph{skew circulant matrix} is a matrix of the form
\begin{equation*}
C\left( c\right) =
\begin{pmatrix}
c_{0} & c_{1} & \ldots  &  & c_{n-1} \\
-c_{n-1} & c_{0} & c_{1} &  & c_{n-2} \\
-c_{n-2} & \ddots  & \ddots  & \ddots  & \vdots  \\
\vdots  & \ddots  & \ddots  & c_{0} & c_{1} \\
-c_{1} & \ldots  & -c_{n-2} & -c_{n-1} & c_{0}%
\end{pmatrix}%
\end{equation*}
\end{definition}

Note that this matrix is again a special case of a Toeplitz matrix. If no confusion arise, the above skew circulant matrix is sometimes  also denoted by $skwcirc(c_0,\ldots,c_{n-1}).$

The next concepts can be seen in  \cite{Karner}. Define the orthogonal (anti-diagonal unit) matrix $J_{m}\in \mathbb{R}^{m\times
m}$ as%
\begin{equation*}
J_{m}:=%
\begin{pmatrix}
0 & 0 & \ldots  & 1 \\
\vdots  &  & 1 & 0 \\
0 &  \rotatebox{45}{$\ldots$}  & \vdots  & \vdots\\
1 & \ldots  & 0 & 0
\end{pmatrix}%
.
\end{equation*}

The matrix
\begin{equation*}
\Gamma _{n}:=\left(
\begin{tabular}{l|l}
$1$ & $\ldots $ \\ \hline
$\vdots $ & $J_{n-1}$
\end{tabular}
\right)
\end{equation*}%
\bigskip is an orthogonal cyclic shift matrix (and a left circulant matrix).

It follows that,
\begin{equation*}
\Gamma _{n}=FF^{T}=F^{2},
\end{equation*}%
where the entries of the unitary discrete Fourier transform (DFT) matrix $F=\left(
f_{pq}\right) $ are given by
\begin{equation*}
f_{pq}:=\frac{1}{\sqrt{n}}\omega ^{pq},\ p=0,1,\ldots ,n-1,\ q=0,1,\ldots
,n-1,
\end{equation*}%
where
\begin{equation*}
\omega =\cos \frac{2\pi }{n}+i\sin \frac{2\pi }{n}=\exp \frac{2\pi i}{n}.
\end{equation*}%
For the orthogonal matrix%
\begin{equation*}
\Xi _{n}=\left(
\begin{tabular}{l|l}
$1$ & $\ldots $ \\ \hline
$\vdots $ & $-J_{n-1}$%
\end{tabular}%
\right)
\end{equation*}%
it is straightforward to verify that
\begin{equation*}
\Xi _{n}=GG^{T},
\end{equation*}%
where $G=\left( g_{pq}\right) $ with
\begin{equation*}
\text{ }g_{pq}=\frac{1}{\sqrt{n}}\omega ^{p\left( q+\frac{1}{2}\right) },\
\ p=0,1,\ldots ,n-1,\ q=0,1,\ldots ,n-1,
\end{equation*}%
is strongly related to the DFT matrix, i.e.,
\begin{equation*}
G=diag\left( 1,\iota ,\ldots ,\iota ^{n-1}\right) F
\end{equation*}%
with
\begin{equation*}
\iota =\omega ^{\frac{1}{2}}.
\end{equation*}%
Therefore, $G$ is also unitary.

Circulant matrices (right or skew right) $M$, $N$  have the following properties (see, for instance \cite{Karner}):
\begin{enumerate}
\item  $M+M$ and $M-N$ are circulant matrices;
\item $M^{T}$ is a circulant\ matrix;
\item  $MN$ is a circulant matrix;
\item $\sum\limits_{\ell
=1}^{k}\alpha _{\ell }M^{\ell }$ is a circulant matrix;
\item A circulant matrix with first row $s$, $S\left(s\right)$ (resp., a skew circulant with first row $c$, $C\left(c\right)$) is diagonalized by the matrix $F$ (resp., by the matrix $G$) thus, the eigenvectors of $S\left(s\right)$ (resp., $C\left(c\right)$) are independent of $s$ (resp., $c$);
\item The eigenvalues of $S\left(s\right)$ (resp., $C(c)$) can be obtained from the first row vector $s\in \mathbb{R}^{n}$ (resp., of $C\left(c\right) \in \mathbb{R}^{n}$).
\end{enumerate}

\begin{remark}
\label{entries}
Let $$S=circ\left(s_0,s_1,\ldots,s_{n-1}\right):=\left(s_{ij}\right)$$ and $$C=skwcirc\left(c_0,c_1,\ldots,c_{n-1}\right):=\left(c_{ij}\right),$$ then

\begin{equation*}
s_{ij}=\left \{
\begin{tabular}{ll}
$s_{j-i }$ & $1\leq i\leq j\leq n$ \\
$s_{n- i+j}$ & $1\leq j<i\leq n$
\end{tabular}\right.
\end{equation*}
and
\begin{equation}
c_{ij}=\left\{
\begin{tabular}{cc}
$c_{j-i}$ & $j\geq i$ \\
$-c_{n-i+j}$ & $i>j.$
\end{tabular}
\right.  \label{cij}
\end{equation}%
\end{remark}


\begin{theorem}
\label{eqcsk}
Consider \[J=%
\begin{pmatrix}
0 & 1 \\
1 & 0%
\end{pmatrix}.
\]
Let
\begin{equation}
S=circ\left(s_0,\ldots, s_{n-1}\right)
\end{equation}\\
and
\begin{equation}
\pm\gamma C=skwcirc\left(\pm\gamma c_0, \pm\gamma c_1,\ldots, \pm\gamma c_{n-1}\right)
\label{skw3}
\end{equation}
be a circulant and a skew circulant matrix, respectively. If
\begin{equation}
\left\vert c_{\iota}\right\vert \leq s_{\iota}, \quad
 \text{for}\ \iota=0,\ldots,n-1 \label{rc}
\end{equation}
 Then, the matrix $M$ obtained in (\ref{aspectM}) from $S$ and $C$ is a nonnegative permutative matrix which takes the form
\begin{equation}
N_\pm\gamma:= \begin{pmatrix}
N_{0} & N_{1} & \ldots  &  &  & N_{n-1} \\
JN_{n-1} & N_{0} & N_{1} & \ldots & \ldots  & N_{n-2} \\
JN_{n-2} & JN_{n-1} & N_{0} & \ldots & \ddots  & \vdots  \\
\vdots  &  & \ddots  & \ddots  &  &  \\
JN_{2}  & \ldots  &  & \\
JN_{1} & \ldots  &  &  & JN_{n-1} & N_{0}
\end{pmatrix}  \label{aspectMcir}
\end{equation}
 where, for $j=0,1,\ldots,n-1$ \[
N_{j}=
\begin{pmatrix}
\frac{s_j+\gamma c_j}{2} & \frac{s_j-\gamma c_j}{2} \\
\frac{s_j-\gamma c_j}{2} & \frac{s_j+\gamma c_j}{2}
\end{pmatrix}.
\]
\end{theorem}
\begin{proof}
This result is a direct consequence of the condition in the statement and of the construction of the matrix $M$ in (\ref{aspectM}).
\end{proof}

To illustrate the above fact we give the following example.
\begin{example}
Let
$$
S=\left(
  \begin{array}{cccc}
2&	4&	0&	2\\
2&	2&	4&	0\\
0&	2&	2&	4\\
4&	0&	2&	2\\
 \end{array}
\right)\,\,\,\, and \,\,\,\,
C=\left(
  \begin{array}{cccc}
-1&	1&	0&	1\\
-1&	-1&	1&	0\\
 0&	-1&	-1&	1\\
-1&  0& -1&-1\\
 \end{array}
\right)
$$
be a circulant and a skew circulant matrix whose spectra, respectively, are $\{8,-4,2+2i,2-2i\}$ y $\{-1+i \sqrt{2},-1+i \sqrt{2},-1-i \sqrt{2},-1-i \sqrt{2}\}$. Both matrices satisfy (\ref{rc}). Thus, the matrix in (\ref{aspectMcir}) becomes

\begin{equation*}
M=\left(
\begin{array}{cccccccc}
\frac{1}{2} &	\frac{3}{2}	&	\frac{5}{2}	&	\frac{3}{2} &	0	& 0		&\frac{3}{2}	&\frac{1}{2}\\
\frac{3}{2}	&\frac{1}{2}&		\frac{3}{2}	&\frac{5}{2}&	0&	0&	\frac{1}{2}	&	\frac{3}{2}\\
\frac{1}{2}	&	\frac{3}{2} &	\frac{1}{2}	&	\frac{3}{2} &	\frac{5}{2}	&	\frac{3}{2}&	0&	0\\
\frac{3}{2}	&\frac{1}{2}	&	\frac{3}{2} &	\frac{1}{2}	&	\frac{3}{2} & \frac{5}{2} &	0&	0\\
0&	0	&\frac{1}{2}&	\frac{3}{2}	& \frac{1}{2}	&\frac{3}{2}&	\frac{5}{2}	&	\frac{3}{2}\\
0&	0 &	\frac{3}{2}&	\frac{1}{2}& \frac{3}{2} &\frac{1}{2} &	\frac{3}{2}	& \frac{5}{2}\\
\frac{3}{2} &	\frac{5}{2}&	0&	0&	\frac{1}{2}	&	\frac{3}{2}&	\frac{1}{2}&	\frac{3}{2}\\
\frac{5}{2}	&	\frac{3}{2}&	0&	0&	\frac{3}{2}	& \frac{1}{2} &	\frac{3}{2}	& \frac{1}{2}
\end{array}
\right)
\end{equation*}

and is a nonnegative permutative matrix with complex spectrum  $$\{8,-4,2+2i,2-2i,-1+i \sqrt{2},-1+i \sqrt{2},-1-i \sqrt{2},-1-i \sqrt{2}\}.$$
\end{example}



The following results characterize the circulant and skew circulant spectra.

\begin{theorem}
\cite{Karner} $S\left( s\right) =F^{\ast }\Lambda \left( s\right) F$, where $$
\Lambda \left( s\right) =diag\left( \lambda _{0}\left( s\right) ,\lambda
_{1}\left( s\right) ,\ldots ,\lambda _{n-1}\left( s\right) \right) $$
and
\begin{equation*}
\text{\ }\lambda _{k}\left( s\right) =\sum\limits_{j=0}^{n-1}s_{j}\omega
^{kj}\text{,\quad\ }k=0,1,\ldots ,n-1.
\end{equation*}
\end{theorem}


\begin{theorem}
\cite{Karner} Let $C\left( c\right) =C^{\ast }M\left( c\right) C$, where $$
M\left( c\right) =diag\left( \mu _{0}\left( c\right) ,\mu _{1}\left(
c\right) ,\ldots ,\mu _{n-1}\left( c\right) \right) $$ and
\begin{equation*}
\text{\ }\mu _{k}\left( c\right) =\sum\limits_{j=0}^{n-1}c_{j}\omega
^{\left( k+\frac{1}{2}\right) j}\text{,\quad\ }k=0,1,\ldots ,n-1.
\end{equation*}
\end{theorem}

\begin{corollary}
\label{fund} Let $$v:=\left( \lambda _{0}\left( s\right) ,\lambda _{1}\left(
s\right) ,\ldots ,\lambda _{n-1}\left( s\right) \right) ^{T}$$ and
$$u=\left(\mu _{0}\left( c\right) ,\mu _{1}\left( c\right) ,\ldots ,\mu _{n-1}\left(c\right) \right) ^{T}.$$
Then, if
\begin{enumerate}
\item $v=\sqrt{n}Fs$, thus
\begin{equation*}
s_{k}=\frac{1}{n}\sum\limits_{j=0}^{n-1}\lambda _{j}\omega ^{-kj}\text{%
,\quad } k=0,1,\ldots ,n-1
\end{equation*}%
and

\item  $u=\sqrt{n}Gc,$ thus
\begin{equation}
c_{k}=\frac{1}{n}\sum\limits_{j=0}^{n-1}\mu _{j}\omega ^{-\left( k+\frac{1}{%
2}\right) j}\text{,\quad }k=0,1,\ldots ,n-1. \label{condicao}
\end{equation}

\end{enumerate}
\end{corollary}

The following results, deal with conjugate symmetry within the spectrum
of $S\left( s\right)$ and $C\left( c\right).$

\begin{theorem}
\label{newmark}\cite{Karner}
\begin{enumerate}
\item $\lambda _{n-k}\left( s\right) =\overline{\lambda _{k}\left( s\right) }
$, for $k=1,2,\ldots ,n-1$ and \ $\lambda _{0}\left( s\right)
=\sum\limits_{j=1}^{n-1}s_{j}.$

\item $\mu _{n-1-k}\left( c\right) =\overline{\mu _{k}\left( c\right) }$,
for $k=0,1,\ldots ,n-1.$
\end{enumerate}
\end{theorem}

\begin{theorem} \label{conditions}
Let $\Lambda =\left\{ \lambda _{0},\lambda _{1},\ldots ,\lambda
_{n-1}\right\} $ and $\Upsilon =\left\{ \mu _{0},\mu _{1},\ldots ,\mu
_{n-1}\right\} $ be the sets such that

\begin{enumerate}
\item $\lambda _{0}=\rho=\max \left\{ \left\vert \lambda _{j}\right\vert
:j=1,2,\ldots ,n\right\}$ and

\item $\lambda _{n-k}=\overline{\lambda _{k}}$, for $k=1,2,\ldots ,n-1$

\item $\mu _{n-1-k}=\overline{\mu _{k}}$, for $k=0,1,\ldots ,n-1,$
\end{enumerate}

and consider the sets
\begin{equation}
\small{\mathcal{P}=\left\{ \alpha\in \mathcal{S}_{n}\text{:\ }\alpha =%
\begin{pmatrix}
0 & 1 & 2 & \cdots  &  & \cdots  & n-1 \\
0 & \ell _{1} & \ell _{2} & \cdots  &  & \cdots  & \ell _{n-1}%
\end{pmatrix}%
;\ \lambda _{n-\ell _{k}}=\overline{\lambda }_{\ell _{k}}, k=1,2,\ldots
,n-1\right\} . } \label{p}
\end{equation}
and
\begin{equation}
\small{\mathcal{Q}=\left\{ \beta \in \mathcal{S}_{n}\text{:\ }\beta =%
\begin{pmatrix}
0 & 1 & 2 & \cdots &  & \cdots & n-1 \\
\ell _{0} & \ell _{1} & \ell _{2} & \cdots &  & \cdots & \ell _{n-1}%
\end{pmatrix}%
;\ \mu _{n-1-\ell _{k}}=\overline{\mu} _{\ell _{k}}, k=0,1,\ldots
,n-1\right\} .}  \label{q}
\end{equation}%
Let $0\leq\gamma\leq1$. A sufficient condition for the lists $\Lambda \cup \left(
\pm \gamma \right) \Upsilon $ to be realized by a permutative matrix $M_{\pm\gamma}$ as in (\ref{aspectM}) is
\begin{equation}
\lambda _{0}\geq \min_{\alpha \in \mathcal{P}}\max_{0\leq k\leq 2m}
\left\{
\begin{array}{c}
-2\sum\limits_{j=1}^{m}\rm{Re} \lambda _{\alpha\left(j\right)}\cos \frac{2k j\pi }{2m+1}- \\
\qquad -2\sum\limits_{j=1}^{m}\rm{Im}\lambda _{\alpha\left(j\right)} \sin\frac{2k j\pi }{2m+1}
\end{array}
\label{ns1}
\right.
\end{equation}
whenever $n=2m-1$
\begin{equation}
\lambda _{0}\geq \min_{\alpha \in \mathcal{P}}\max_{0\leq k\leq 2m+1}
\left\{
\begin{array}{c}
-2\sum\limits_{j=1}^{m-1}\rm{Re}\lambda _{\alpha\left(j\right)}\cos \frac{
2k j\pi }{m+1}-\left( -1\right) ^{k}\lambda _{m}- \\
\qquad -2\sum\limits_{j=1}^{m-1}\rm{Im}\lambda _{\alpha\left(j\right)}
\sin \frac{2k j\pi }{m+1}%
\end{array}%
  \label{ns2}
  \right.
\end{equation}
whenever $n=2m-2,\ $and there exists $\left( \alpha ,\beta \right) \in \mathcal{P}%
\times \mathcal{Q}$, for all $k=0,1,\ldots ,n-1$ such that the following inequalities hold

\begin{equation}
\sum\limits_{j=0}^{n-1}\lambda _{\alpha \left( j\right) }\omega ^{-kj}\geq
\left\vert \sum\limits_{j=0}^{n-1}\mu _{_{\beta \left( j\right) }}\omega
^{-\left( k+\frac{1}{2}\right) j}\right\vert \text{.\quad }
\label{c1}
\end{equation}
\end{theorem}
\begin{proof}
A set $\Lambda =\left\{ \lambda _{0},\lambda _{1},\ldots ,\lambda
_{n-1}\right\} $ satisfying the conditons $1$. and $2.\ $in the statement is
called \emph{even conjugate }and by Theorem 4 in \cite{RSGuo} a necessary and
sufficient condition for $\Lambda $ to be the spectrum of a real circulant
matrix are given by (\ref{ns1}) and (\ref{ns2}). On the other hand, let us
consider the set of permutations (\ref{q})$.\ $For $\beta \in \mathcal{Q}$
the skew circulant matrix whose first row is given by
\begin{equation}
c_{\beta} =\frac{1}{\sqrt{n}}G^{\ast}\beta\left( u \right) \text{,}  \label{une}
\end{equation}
(where $\beta \left( u\right) $ is considered as in Definition \ref{ept}) is a
real skew circulant matrix and all the components of the vector $%
u=\left( \mu _{0},\mu _{1},\ldots ,\mu _{n-1}\right) $ belongs to its spectrum. The
conditions in (\ref{c1}) reflect the conditions in (\ref{mayorize})  as to  obtain a nonnegative matrix with the shape in  (\ref{aspectMcir}) it is clear that it is enough to compare only the first row of the matrices $S$ and $\left |C \right |= (\left| c_{ij} \right|)$.
\end{proof}

\begin{corollary}
Let $\Lambda =\left\{ 1,\lambda _{1},\ldots ,\lambda _{n-1}\right\} $ and $%
\Upsilon =\left\{ \mu _{0},\mu _{1},\ldots ,\mu _{n-1}\right\} $ be two lists
satisfying the conditons:
\begin{enumerate}
\item $1=\max \left\{ \left\vert \lambda _{j}\right\vert :j=1,2,\ldots
,n\right\} $

\item $\lambda _{n-k}=\overline{\lambda _{k}}$, for $k=1,2,\ldots ,n-1$

\item $\mu _{n-1-k}=\overline{\mu _{k}}$, for $k=0,1,\ldots ,n-1,$
\end{enumerate}
and consider the sets\ $\mathcal{P}$ and $\mathcal{Q}$ as in (\ref{p}) and (%
\ref{q}), respectively.\ Let $0\leq \gamma \leq 1$.\ A sufficient condition
for $\Lambda \cup \left( \pm \gamma \right) \Upsilon $ to be the spectrum of
a permutative nonnegative matrix is:

\begin{equation*}
\small{ 1\geq \min_{\alpha\in \mathcal{P}}\max_{0\leq k\leq 2m}\left\{
\begin{array}{l}
-2\sum\limits_{j=1}^{m}\rm{Re}\lambda _{\alpha\left(j\right)} \cos \frac{%
2kj\pi }{2m+1}-2\sum\limits_{j=1}^{m}\rm{Im}\lambda _{\alpha\left(j\right)}\sin
\frac{2kj\pi }{2m+1}%
\end{array}
\right.}
\end{equation*}

whenever $n=2m-1$ and

\begin{equation*}
\small{ 1\geq \min_{\alpha \in \mathcal{P}}\max_{0\leq k\leq 2m+1}\left\{
\begin{array}{l}
-2\sum\limits_{j=1}^{m-1}\rm{Re}\lambda _{\alpha\left(j\right)}\cos \frac{%
2kj\pi }{m+1}-\left( -1\right) ^{k}\lambda _{m} -2\sum\limits_{j=1}^{m-1}\rm{Im}\lambda _{\alpha\left(j\right)}
\sin \frac{2kj\pi }{m+1}
\end{array}
\right.}
\end{equation*}

whenever $n=2m-2. $ \\
Moreover, there exists  $\left( \alpha ,\beta \right) \in \mathcal{P}%
\times \mathcal{Q},$ for all $k=0,1,\ldots ,n-1$, such that the following inequalities hold

\begin{equation*}
\left\vert \sum\limits_{j=0}^{n-1}\mu _{_{\beta \left( j\right) }}\omega
^{-\left( k+\frac{1}{2}\right) j}\right\vert \leq
1+\sum\limits_{j=1}^{n-1}\lambda _{\alpha \left( j\right) }\omega ^{-kj}\text{%
\quad }.
\end{equation*}
Thus,
\begin{equation*}
\small{\bigcup_{\left(\alpha,\beta\right)\in \mathcal{P} \times \mathcal{Q}}
\{ \left(\lambda _{\alpha\left(1\right)},\ldots ,\lambda _{\alpha\left(n-1\right)}, \mu _{\beta\left(0\right)},\ldots , \mu _{\beta\left(n-1\right)}\right): \left\vert \sum\limits_{j=0}^{n-1}\mu _{_{\beta \left( j\right) }}\omega
^{-\left( k+\frac{1}{2}\right) j}\right\vert \leq
1+\sum\limits_{j=1}^{n-1}\lambda _{\alpha \left( j\right) }\omega ^{-kj} \}}
\end{equation*}

is a complex permutative realizability region for
 $$\left\{ 1,\lambda _{1},\ldots ,\lambda _{n-1},\mu _{0}, \mu _{1},\ldots , \mu _{n-1}\right\} $$ in terms of $\left\{
\lambda _{1}.\ldots ,\lambda _{n-1}\right\}.
$ In particular for
\begin{equation*}
r=\min_{\alpha\in \mathcal{P}}\left\{ \sum\limits_{j=1}^{n-1}\lambda _{\alpha
\left( j\right) }\omega ^{-kj}:k=0,1,\ldots ,n-1\right\}
\end{equation*}

the set
\begin{equation*}
\small{\bigcup_{\beta\in \mathcal{Q}}
\{ \left(\mu _{\beta\left(0\right)},\ldots , \mu _{\beta\left(n-1\right)}\right): \left\vert \sum\limits_{j=0}^{n-1}\mu _{_{\beta \left( j\right) }}\omega
^{-\left( k+\frac{1}{2}\right) j}\right\vert \leq 1+r,\quad k=0,1,\ldots
,n-1 \}}
\end{equation*}
is also a complex permutative realizability region in terms of $\left\{
\lambda _{1}.\ldots ,\lambda _{n-1}\right\} .$
\end{corollary}

\begin{remark}
By the trace property, for all $\left( \alpha ,\beta \right) \in \mathcal{P}\times \mathcal{Q}$ the constant diagonal elements of the matrices $C=skwcirc(c_{\beta}^{T})$ and $S=circ(s_{\alpha}^{T})$ obtained from the vectors $c_{\beta}$ in (\ref{une}) and
\begin{equation}
s_{\alpha}= \frac{1}{\sqrt{n}}F^{\ast}\alpha\left( v \right) \text{,}  \label{due}
\end{equation}
(where $\alpha\left( v\right) $ is considered as in Definition \ref{ept})
coincide, respectively.
\end{remark}

\begin{theorem}
\cite{MAR}\label{main2 copy(2)} Let $S=\left( s_{ij}\right) $ be a nonnegative matrix of order $%
 n+1$ and consider the  $C=skwcirc\left(c_0,c_1,\ldots,c_{n-1}\right):=(c_{ij})$ whose spectra (counted with their
multiplicities) are $\left\{ \lambda _{0},\lambda _{1},\ldots ,\lambda _{n}\right\} $ and $\left\{\mu _{0},\mu _{1},\ldots ,\mu _{n-1}\right\} $,
respectively.\ Moreover, suppose that $
|c_{ij}|\leq s_{ij}$, \  $1\leq i,j \leq n. $  Then the nonnegative matrix
\begin{equation}
M=%
\begin{pmatrix}
\frac{s_{11}\pm \gamma c_{11}}{2} & \frac{s_{11}\mp \gamma c_{11}}{2} & \ldots & \ldots &
\frac{s_{1n}\pm \gamma c_{1n}}{2} & \frac{s_{1n}\mp \gamma c_{1n}}{2} & s_{1,n+1} \\
\frac{s_{11}\mp \gamma c_{11}}{2} & \frac{s_{11}\pm \gamma c_{11}}{2} & \ldots & \ldots &
\frac{s_{1n}\mp\gamma c_{1n}}{2} & \frac{s_{1n}\pm \gamma c_{1n}}{2} & s_{1,n+1} \\
\vdots & \vdots & \ddots & \ddots & \vdots & \vdots & \vdots \\
\vdots & \vdots & \ddots & \ddots & \vdots & \vdots & \vdots \\
\frac{s_{n1}\mp \gamma c_{n1}}{2} & \frac{s_{n1}\pm \gamma c_{n1}}{2} & \ldots & \ldots &
\frac{s_{nn}\pm \gamma c_{nn}}{2} & \frac{s_{nn}\mp \gamma c_{nn}}{2} & s_{n,n+1} \\
\frac{s_{n1}\pm \gamma c_{n1}}{2} & \frac{s_{n1}\mp \gamma c_{n1}}{2} & \ldots & \ldots &
\frac{s_{nn}\mp\gamma c_{nn}}{2} & \frac{s_{nn}\pm \gamma c_{11}}{2} & s_{n,n+1} \\
s^{1}_{n+1,1} & s^{2}_{n+1,1} & \ldots & \ldots & s^{1}_{n+1,n}%
& s^{2}_{n+1,n} & s_{n+1,n+1}%
\end{pmatrix}%
,  \label{matrixM2}
\end{equation}
where $$ s^{1}_{n+1,i}+s^{2}_{n+1,i}=s_{n+1,i} \qquad 1\leq i \leq n$$
realizes the list
\begin{equation*}
\left\{\lambda _{0},\lambda _{1},\ldots ,\lambda _{n},\pm \gamma\mu
_{0},\pm \gamma\mu _{1},\ldots ,\pm \gamma \mu _{n-1}\right\}.
\end{equation*}
\end{theorem}





From the above facts the following definition can be stated.
\begin{definition}
Given the sets $\Upsilon =\left\{ \mu _{0},\mu _{1},\ldots ,\mu
_{n-1}\right\} $ and $\Lambda =\left\{ \rho ,\lambda _{1},\lambda
_{2},\ldots ,\lambda _{n}\right\} $ we say that $\Upsilon $ (resp. $%
\Lambda $) is skew circulant (resp. circulant) spectrum if the condition $3$.\
(resp. $1$ and $2.$)\ in Theorem \ref{conditions} holds.
\end{definition}
The next problem can be formulated:

\begin{problem} \label{P1}
Given the skew circulant spectrum $\Upsilon =\left\{ \mu
_{0},\mu _{1},\ldots ,\mu _{n-1}\right\}, $ under which conditions does there
exist a realizable circulant spectrum $\Lambda =\left\{ \rho ,\lambda
_{1},\ldots ,\lambda _{n}\right\} $
 such that $\Lambda \cup \pm \gamma \Upsilon $ is realizable,
for all $\gamma \in \left[ 0,1\right]? $
\end{problem}

In order to  give an answer to this problem, we need to recall the following facts:

\begin{enumerate}
\item From formula (\ref{condicao}) at item 2. in Corollary \ref{fund} and Theorem \ref{conditions}, the following
inequalities can be easily obtained
\begin{equation*}
\left\vert c_{\ell}\right\vert \leq\max_{\beta\in\mathcal{Q}, \
0\leq k\leq n-1} \frac{1}{n}
\left\vert \sum\limits_{j=0}^{n-1}\mu _{_{\beta \left( j\right) }}\omega
^{-\left( k+\frac{1}{2}\right) j}\right\vert \  \text{,\quad }\ell=0,1,\ldots ,n-1.
\end{equation*}

\item Let
\begin{equation*}
\chi= \max_{\beta\in\mathcal{Q}, \
0\leq k\leq n-1} \frac{1}{n}
\left\vert \sum\limits_{j=0}^{n-1}\mu _{_{\beta \left( j\right) }}\omega
^{-\left( k+\frac{1}{2}\right) j}\right\vert  \label{chi}
\end{equation*}

\item Eqs. (\ref{ns1}) and (\ref{ns2}) give necessary and sufficient conditions for
$$\widetilde{\Lambda }=\left\{ \rho -\left( n+1\right) \chi,\lambda _{1},\lambda _{2},\ldots ,\lambda _{n}\right\} $$ to be the spectrum
of a nonnegative circulant matrix $B=\left( b_{ij}\right) $. Furthermore, by
Perron Frobenius Theory (see \cite{BergmanPlemmons}), $B$ is an irreducible nonnegative matrix
and the positive $(n+1)$ -dimensional eigenvector $\mathbf{e}=\left(
1,\ldots ,1 \right) ^{T}$ is associated to the eigenvalue $\rho
-\left( n+1\right) \chi $ of $B$.
\item By Brauer Theorem (see \cite{RS}) the matrix $R=B+\chi \mathbf{ee}^{T}$ has
spectrum
$$\widetilde{\Lambda }\setminus \left\{ \rho -\left( n+1\right) \chi
\right\} \cup \left\{ \rho -\left( n+1\right) \chi +\chi \mathbf{e}^{T}%
\mathbf{e=}\rho \right\} =\Lambda .$$

\item Moreover, for the $\left( i,j\right) $-entry of $R:=\left(
r_{ij}\right) $ we have $r_{ij}=b_{ij}+\chi \geq \chi \geq \left\vert
c_{k}\right\vert, \quad k=0,1,\ldots ,n-1.$

\item By Theorem \ref{main2 copy(2)} a nonnegative matrix of order $2n+1$ of the form of $M$ in (\ref{matrixM2}) can be constructed from the matrices $R$ and $skwcirc\left(c_0,c_1,\ldots,c_{n-1}\right).$

\end{enumerate}

In consequence, the following result can be stated.

\begin{proposition}
Let $\Upsilon$ be the skew circulant spectrum $\Upsilon =\left\{ \mu _{0},\mu
_{1},\ldots ,\mu _{n-1}\right\} $. If there exists a nonnegative circulant
matrix with spectrum $$\widetilde{\Lambda }=\left\{ \rho -\left( n+1\right)
\chi ,\lambda _{1},\lambda _{2},\ldots ,\lambda _{n}\right\},$$ then there
exists a nonnegative matrix with spectrum $\Lambda \cup \pm \gamma \Upsilon $
where $\gamma \in \left[ 0,1\right] $ and $$\Lambda =\left\{ \rho ,\lambda _{1},\lambda _{2},\ldots ,\lambda _{n}\right\} .\ $$
In particular, if $\lambda_{\iota}=0, \iota=1,\ldots,n-1$ the matrix $M$ in (\ref{matrixM2}) obtained from $C$ and the rank one matrix $$S=circ(\frac{\rho}{n+1},\ldots,\frac{\rho}{n+1})$$ is permutative.
\end{proposition}

The following example shows that the condition in the above proposition can be weakened.
\begin{example}
The spectrum $\Omega =\left\{ 15,1,7,2+5i,2-5i,\frac{5+i\sqrt{3}
}{2},\frac{5-i\sqrt{3}}{2}\right\} $ can be partitioned into the
circulant spectrum $\Lambda =\left\{ 15,1,2+5i,2-5i\right\} $ and the skew
circulant spectrum $\Upsilon =\left\{ 7,\frac{5+i\sqrt{3}}{2}+\frac{5-i\sqrt{3}}{2},
\right\} $.\ The first and the second one
are realized by the circulant matrix $S$ and the skew circulant matrix $C$, respectively:
\[
S=%
\begin{pmatrix}
5 & 6 & 3 & 1 \\
1 & 5 & 6 & 3 \\
3 & 1 & 5 & 6 \\
6 & 3 & 1 & 5%
\end{pmatrix}%
, \quad \text{ }C=%
\begin{pmatrix}
4 & -2 & 1 \\
-1 & 4 & -2 \\
2 & -1 & 4%
\end{pmatrix}%
.
\]%
So, the union  $\Omega$ is realized by the nonnegative matrix
\[
M=
\begin{pmatrix}
\frac{9}{2} & \frac{1}{2} & 2 & 4 & 2 & 1 & 1 \\
\frac{1}{2} & \frac{9}{2} & 4 & 2 & 1 & 2 & 1 \\
0 & 1 & \frac{9}{2} & \frac{1}{2} & 2 & 4 & 3 \\
1 & 0 & \frac{1}{2} & \frac{9}{2} & 4 & 2 & 3 \\
\frac{5}{2} & \frac{1}{2} & 0 & 1 & \frac{9}{2} & \frac{1}{2} & 6 \\
\frac{1}{2} & \frac{5}{2} & 1 & 0 & \frac{1}{2} & \frac{9}{2} & 6 \\
3 & 3 & 3 & 0 & 1 & 0 & 5%
\end{pmatrix}.%
\]
\end{example}
The next definition generalizes the definitions of circulant and skew ciculant matrices.

\begin{definition}
\label{abscir}
Let $C$ be a matrix of order $n$, we will say that $C=\left(c_{ij}\right)$ is an \textit{absolutely circulant matrix} if the absolute value matrix of $C$, $\left\vert C \right\vert:=\left(\left\vert c_{ij}\right\vert\right)$ is circulant, being the diagonal element either non positive or all non negative.
If $C$ is an absolutely circulant matrix with first row $\left(c_{0},c_{1},\ldots ,c_{n-1}\right)$ we write,$$C=abscirc\left(c_{0},c_{1},\ldots ,c_{n-1}\right).$$
\end{definition}

\begin{lemma}\label{circ2}
Let $C=abscirc\left(c_{0},c_{1},\ldots ,c_{n-1}\right):=\left(c_{ij}\right) $. Then
\begin{equation}
c_{ij}=\left\{
\begin{tabular}{cc}
$\pm c_{j-i}$ & $j\geq i$ \\
$\pm c_{n-i+j}$ & $i>j.$%
\end{tabular}%
\right.   \label{cij2}
\end{equation}%
\end{lemma}

\begin{proof}
Is a direct consequence of the Definition \ref{abscir}.
\end{proof}
\begin{remark}
Lemma \ref{circ2} really takes part of the definition of an absolutely circulant matrix as it gives the change of the sign (or not) of the entry under the diagonal with respect to its pair up the diagonal.
\end{remark}

The example below gives an absolutely circulant matrix which is neither circulant nor skew circulant.

\begin{example}
Let us consider
$$
C=\left(
\begin{array}{ccc}
1&-2&3\\
3&1&2\\
-2&-3&1\\
\end{array}
\right).
$$
It is clear that $C$ is an absolutely circulant matrix which is neither circulant nor skew circulant. In fact, $|C|=circ(1,2,3)$.
\end{example}

\begin{proposition}
Let $S=circ(s_0,s_1,\ldots,s_{n-1})$ and $C=abscirc\left(c_{0},c_{1},\ldots ,c_{n-1}\right)$ be a circulant and an absolutely circulant matrix, respectively such that $$\left\vert c_{\iota}\right \vert \leq s_{\iota} \ \iota=0,\ldots, n-1. $$ Then the matrix $M$ in (\ref{aspectM}) is a nonnegative permutative matrix.
\end{proposition}
\begin{proof}
This proof is constructive and it is an analogous proof to the one presented at Theorem \ref{eqcsk}.
\end{proof}




 From the above definition and properties we now formulate the following question:

\begin{problem}
Could be obtained a spectral characterization for an absolutely circulant matrix in terms of matrices related to Discrete transform Fourier?
\end{problem}

\textbf{Acknowledgments}.
Enide Andrade was supported in part by the Portuguese Foundation for Science and Technology (FCT-Funda\c{c}\~{a}o para a Ci\^{e}ncia e a Tecnologia), through CIDMA - Center for Research and Development in Mathematics and Applications, within project UID/MAT/04106/2013. M. Robbiano was partially supported by project VRIDT UCN 170403003.


\begin{thebibliography}{99}
\bibitem{BergmanPlemmons} A. Berman, R. J. Plemmons, Nonnegative Matrices in the Mathematical Sciences, Series: Classics in Applied Mathematics, SIAM, 1994.

\bibitem{Boro} A. Borobia, On nonnegative eigenvalue problem, Lin. Algebra
Appl. 223/224 (1995), 131-140, Special Issue honoring Miroslav Fiedler and
Vlastimil Pt\'{a}k.

\bibitem{BoSo} A. Borobia, J. Moro, R. L. Soto, Negativity compensation in the nonnegative, Linear Algebra Appl. 393 (2004), $73-89$.

\bibitem{Boyle} M. Boyle, D. Handelman, The spectra of nonnegative matrices
via symbolic dynamics, Ann. of Math. (2) 133 (1991) 2,\ $249-316.$



\bibitem{Davis} D. Philip, Circulant matrices, JOHN WILEY \& SONS, New York, Chichester, Brisbane, Toronto (1979).

\bibitem{EG} P. D. Egleston, T. D. Lenker, S. K. Narayan, The nonnegative inverse eigenvalue problem, Linear Algebra and its Applications 379 (2004) 475–490

\bibitem{Fiedler} M. Fiedler, Eigenvalues of nonnegative symmetric matrices, Lin. Algebra Appl. 9 (1974), 119-142.

\bibitem{Friedland} Sh. Friedland, On an inverse problem for nonnegative and
eventually nonnegative matrices, Israel T. Math. 29 (1978), 1, 43-60.


\bibitem{LealJonhson} A. Leal- Duarte, C.R. Johnson, Resolution of the symmetric nonnegative inverse eigenvalue problem for matrices subordinate to a bipartite graph, Positivity 8 (2004): 209-2013.

\bibitem{Hu} X. Hu, C. R. Johnson, C. E. Davis, and Y. Zhang. Ranks of permutative matrices, Spec. Matrices, 4, (2016): 233-246.

\bibitem{JhonsonMarijuan} C. R. Johnson, C. Mariju\'an, P. Paparella, M. Pisonero, The NIEP, https://arxiv.org/pdf/1703.10992.pdf (2017).

\bibitem{Johnson} C. R. Johnson, Row stochastic matrices similar to doubly
stochastic matrices, Linear and Multilinear Algebra (1981), 2, 113-130.

\bibitem{Laffey} C. R. Johnson, T. Laffey, R. Loewy, The real and symmetric
nonnegative inverse eigenvalue problems are different, Proc. Amer. Math
Soc., 124 (1996), 12,\ 3647-3651.

\bibitem{Laffey2} T. Laffey, Extreme nonnegative matrices, Lin. Algebra
Appl. 275/276 (1998), 349-357. Proceedings of the sixth conference of the
international Linear Algebra Society (Chemnitz, 1996).

\bibitem{Karner} Karner, H., Schneid, J., Ueberhuber, C. W. Spectral decomposition of real circulant matrices. Lin. Algebra Appl. 367 (2003):301-311.

\bibitem{KiHang} Ki Hang Kim, S. N. Ormes, W. F. Roush, The spectra of
nonnegative integer matrices via formal power series, J. Amer. Math. Soc. 13
(2000), 4, 773-806 (electronic)


\bibitem{muitos} T. Laffey, H. \v{S}migoc, Realizing matrices in the nonnegative inverse
eigenvalue problem, Matrices and group representations (Coimbra, 1998),
Textos Mat. S\'{e}r. B, vol 19, Univ. Coimbra, Coimbra, 1999, pp. 21-31.


\bibitem{kolmogorov} A. N. Kolmogorov, Markov chains with a countable number of possible states, Bull. Moskow Gosu-darstvennogo Univ. Mat. Meh., 1(3)(1937) 1-16.



\bibitem{LaffSmgc2} T. Laffey, H. \v{S}migoc, Nonnegative realization of spectra having negative real parts, Lin. Algebra Appl. 416 (2006), 1, 148-159.

\bibitem{LMc} R. Loewy, J. J. Mc Donald, The symmetric nonnegative inverse
eigenvalue problem for $5\times 5$ matrices, Lin. Algebra Appl. 393
(2004), 275-298.

\bibitem{LwyLdn} R. Loewy, D. London, A note on an inverse problem for
nonnegative matrices, Linear and Multilinear Algebra 6 (1978/79) 83-90.


\bibitem{Lwy} R. Loewy. A note on the real nonnegative inverse eigenvalue problem, Electron. J. Linear Algebra, 31 (2016): 765-773.

\bibitem{MAR} C. Manzaneda, E. Andrade, M. Robbiano. Realizable lists via the spectra of structured matrices, Lin. Algebra Appl. (2017). Accepted.

\bibitem{Mayo} J. Torre Mayo, M. R.  Abril Raymund, E. Alarcia Estévez, C. Marijuán, M. Pisonero, The nonnegative inverse problema from the coeficientes of the characteristic polynomial EBL digraphs, Lin. Algebra Appl. 426, (2007): 729-773.

\bibitem{Meehan} M. E.  Meehan, Some results on matrix spectra, Phd thesis, National University of Ireland, Dublin, 1998.

\bibitem{Mink} H. Mink, Non-negative Matrices, John Wiley and Sons, New York, 1988.

\bibitem{Oliveira} G. N. Oliveira, Sobre matrizes estocásticas e duplamente estocásticas, PhD Thesis, (1968), Coimbra.

\bibitem{PP} P. Paparella, Realizing Suleimanova-type Spectra
via Permutative Matrices, Electron. Journal of Linear Algebra, Volume 31, (2016) 306-312.

\bibitem{HzlP} H. Perfect, On positive stochastic matrices with real
characteristic roots, Proc. Cambridge Philos. Soc. 48 (1952): 271-276.


\bibitem {RojoSoto} O. Rojo, H. Soto, Some results on symmetric circulant matrices and on symmetric centrosymmetric matrices, Lin. Algebra Appl., 392, (2004): 211-233.

\bibitem{RSGuo} O. Rojo, R. L. Soto. Guo perturbations for symmetric nonnegative circulant matrices, Lin. Algebra Appl. 431 (2009): 594-607.

\bibitem{RS} O. Rojo, R. L. Soto. Applications of a Brauer Theorem in the nonnegative inverse eigenvalue problem, Lin. Algebra Appl. 416 (2007): 1-18.

\bibitem{SRM} R. Soto, O. Rojo, C. Manzaneda. On the nonnegative realization of partitioned spectra, Electron. Journal of Linear Algebra, 22, (2011): 557-572.

\bibitem{SmgcH} H. \v{S}migoc, The inverse eigenvalue problem for
nonnegative matrices, Lin. Algebra Appl. 393 (2004): 365-374.

\bibitem{NN3}  H. \v{S}migoc. Construction of nonnegative matrices and the inverse
eigenvalue problem, Lin. and Multilin. Algebra 53, 2 (2005): 85-96.

\bibitem{Soules} G. W. Soules,  Constructing symmetric nonnegative matrices,
Linear and Multilinear Algebra 13-3 (1983): 241-251.

\bibitem{SLMNva} H. R. Sule\u{\i}manova, Stochastic matrices with real
characteristic numbers, Doklady, Akad. Nuk SSSR (N. S.) 66 (1949):343-345.

\bibitem{GuoWen} G. Wuwen, Eigenvalues of nonnegative matrices, Lin.
Algebra Appl. 266 (1997):261-270.

\end{thebibliography}
\end{document}